\documentclass[11pt,oneside,a4paper]{elsarticle}

\usepackage[a4paper,top=3cm,bottom=2cm,left=3cm,right=3cm,marginparwidth=1.75cm]{geometry}
 \usepackage{amsmath}\usepackage{etex}
\usepackage{amssymb}
\usepackage{amsfonts,amsthm,cases,mathtools}
\usepackage{threeparttable}
\usepackage{booktabs}
\usepackage{multirow}
\usepackage{stmaryrd}
\usepackage{chemarrow}
\usepackage[norelsize]{algorithm2e}
\usepackage{enumerate}
\usepackage{graphicx}
\usepackage[all]{xy}
\usepackage{tikz}
\usetikzlibrary{arrows}
\usepackage{multirow}
\usepackage{caption}\usepackage{subcaption}
\usepackage[colorinlistoftodos]{todonotes}
\usepackage[colorlinks=true, allcolors=blue]{hyperref}
\usepackage{bm}
\usepackage{mathrsfs}
\newtheorem{theorem}{Theorem}[section]
\newtheorem{lemma}[theorem]{Lemma}
\newtheorem{corollary}[theorem]{Corollary}

\numberwithin{equation}{section}

\newcommand{\mcal}{\mathcal}

\begin{document}

\begin{frontmatter}
  \title{ A robust and optimal method for a sixth-order \\elliptic singular perturbation problem.}
  		\author[zjnu]{Mingqing Chen}
\address[zjnu]{School of Mathematical Sciences, Zhejiang Normal University, Jinhua 321004, China}
  \ead{chenmq@zjnu.edu.cn}
	\fntext[cmqfootnote]{This work of this author was partially supported by NSFC Tianyuan Mathematics Young Researcher Project (Grant No.\ 12526578) and Zhejiang Provincial Natural Science Foundation of China (Grant No.\ LQN26A010006).} 
	
    \author[sjtu]{Jianguo Huang\fnref{hjgfootnote}}
\address[sjtu]{School of Mathematical Sciences, and MOE-LSC, Shanghai Jiao Tong University, Shanghai 200240, China}
	\ead{jghuang@sjtu.edu.cn}
	\fntext[hjgfootnote]{Corresponding author. The work of this author was partially supported by NSFC (Grant No.\ 12571390).}
	
	\begin{abstract}
		In this paper we propose a robust nonconforming finite method for a gradient-elastic Kirchhoff plate (GEKP) model with a small parameter $0<\iota<1$. While the method employs an $H^3$ cubic finite element for the discrete space, it employs an interpolation mapping into an $H^2$ nonconforming finite element space in both the lower-order terms of the bilinear form and the right-hand side. Notably, the numerical method is robust with respect to the parameter $\iota$ and uniformly convergent to order $h$ without any stabilized terms. Finally, some numerical experiments are performed to verify the theoretical findings.
	\end{abstract}
\end{frontmatter}

 \section{Introduction}

 The classical theory continuum theory neglects the microstructure of materials-such as high-strength metal alloys, polymeric foams, and porous bone, and focuses solely on their macroscopic behavior, thereby exhibiting very limited capability in describing multiscale phenomena. To capture the size-dependent mechanical behavior in micro- and nano-structures, higher-order continuum theories have been systematically developed by numerous researchers. The corresponding constitutive laws incorporate internal length scale parameters. We refer the reader to \cite{Cosserat1909,Toupin1962Elastic,MindlinTiersten1962Effects,Koiter1964couplestress,Mindlin1964Micro,Aifantis1984microstructural,AltanAifantis1992,RuAifantis1993simple} for details along this line. These theories remain challenging to implement in engineering practice due to the large number of material constants beyond the two Lam\'{e} parameters. Consequently, most studies adopt only one or two additional constants. Aifantis and his collaborators proposed in \cite{AltanAifantis1992} and \cite{RuAifantis1993simple} an one-parameter simplified strain gradient elasticity (SGE) theory for material deformation at the micro/nano scale, which is effective in simulating the deformation of elastic plastic materials. Moreover, making use of the Aifantis' SGE model and the Kirchhoff-Love hypothesis (cf. \cite{Timoshenko1959}), one can deduce a sixth-order governing equation for the gradient-elastic Kirchhoff plate (GEKP) model (cf. \cite{NiiranenNiemi2017GEK,papargyri2008static}):
 \begin{equation}\label{model:GEK}
 	\begin{cases}
 		\mathcal{D}(  \Delta^2w - \iota^2 \Delta^3 w) = f & \textrm{in}~\Omega ,\\
 		w = \partial_{\bm{n}} w = \partial_{\bm{nn}} w = 0 &\textrm{on}~\partial\Omega,
 	\end{cases}
 \end{equation}
where $\Omega\subset\mathbb R^2$ is a bounded polygon, $\partial_{\bm n}w$ is the normal derivative of $w$, and $\iota\in(0,1)$ is a real small size parameter. Without loss of generality, set ${\mathcal D}=1$ throughout this paper.

 Classical conforming finite element methods for problems requiring high regularity typically involve high-degree polynomials. To circumvent this difficulty, nonconforming finite element methods are widely employed for high-order elliptic problems (cf. \cite{WuXu2019,HuZhang2019,HuZhang2017,WangXu2013,LiWu2024}). Moreover, the $C^0$ interior penalty method is a viable option. See \cite{Gudi2011sixthIPDG,ChenLiQiu2022,ChenHuangHuang2026GEKP} for details. On the other hand, if  $\iota=0$, then the original problem will degenerate to a fourth-order elliptic problem. This implies that the solution of the reduced problem may not meet all of the boundary conditions of problem \eqref{model:GEK} and this would lead to the phenomenon of boundary layer as $\iota$ goes to zero. In order to develop a robust and optimal numerical method with respect to $\iota$, one common way is to enforce the boundary condition $\partial_{\bm nn}w=0$ weakly, for instance by means of Nitsche's method \cite{Nitsche1971}, or penalty techniques \cite{Arnold1982}. A robust $C^0$ interior penalty method for a GEKP model over polygons is proposed in \cite{ChenHuangHuang2026GEKP} with weakly imposed two boundary conditions. Although the proposed method has optimal convergence rate and its discrete space employs only cubic polynomials, the inclusion of stabilization terms renders the scheme relatively complex in practical implementation. Another approach is to introduce an interpolation from the discrete space to another suitable finite element space in both the bilinear form associated with the lower-order differential operator and the right-hand side term, as in the discrete form of \cite{WangXuHu2006}. The authors in \cite{WangXuHu2006} proposed a nonconforming finite element method using a modified Morley element and the Morley element for a fourth-order singular perturbation problem. It should be noted that the finite element method does not require any stabilization terms and converges uniformly with respect to the small parameter.%

  Following the idea in \cite{WangXuHu2006}, this paper presents an $H^3$ nonconforming finite element method \eqref{discreteForm:GEK} by employing the $H^3$ cubic nonconforming finite element space $V_h$ in \cite{HuZhang2019}. But, we use a quasi interpolation of the finite element function in the bilinear form associated with the operator $\Delta^2$ and the right-hand side term. Moreover, the quasi interpolation operator $I_h^{NZT}$ is constructed from $V_h$ into the new Zienkiewicz-type (NZT) nonconforming finite element space designed for problem \eqref{weakForm:KPB} in \cite{WangShiXu2007NZT}. Furthermore, by adapting the analysis framework of \cite{Guzman2012Nitsche,ChenHuangHuang2025} and the regularity result in \cite{ChenHuangHuang2026GEKP}, we can derive the optimal and uniform error estimates with respect to the size parameter.

The rest of this paper is organized as follows. In Sect. \ref{SectNotation}, we introduce some notation and existing results for later analysis. In Sect. \ref{SectDiscreteF}, we design a nonconforming finite element method of problem \eqref{model:GEK} with the help of an interpolation operator. In Sect. \ref{Sect:ErrorAnalysis}, we analyse the convergence properties of the proposed method. In the last section, a numerical experiment is performed to verify the theoretical findings.

 \section{Preliminaries} \label{SectNotation}
  Throughout this paper, let $\Omega\subset\mathbb R^2$ be a bounded polygon, which is partitioned into a family of shape regular triangles $\mathcal{T}_h=\{K\}$. That is to say that there exists a constant $\gamma_0\geq 1$ such that
 \[
 \frac{h_K}{\rho_K}\leq \gamma_0 \quad \forall \;K\in \mcal T_h,
 \]
 where $h_K={\rm diam}(K)$ and $h=\max_{K\in \mathcal{T}_h}h_K$. Use $\mathcal V(\mathcal T_h)$ and $\mathcal V^i(\mathcal T_h)$ (resp. $\mathcal{E}(\mathcal T_h)$ and $\mathcal{E}^i(\mathcal T_h)$) to stand for the sets of all and interior vertices (resp. edges) of $\mathcal T_h$, respectively.
 We then introduce the patch of vertices and edges for further use.For any vertex $\delta\in \mathcal V(\mathcal T_h)$ and edge $e\in \mathcal{E}(\mathcal T_h)$, write $\omega_{\delta}$ and $\omega_{e}$ to be respectively the unions of all triangles in $\mathcal{T}_{\delta}$ and $\mathcal{T}_{e}$, where $\mathcal{T}_{\delta}$ and $\mathcal{T}_{e}$ is the set of all triangles in $\mathcal{T}_h$ sharing common vertex $\delta$ and edge $e$, respectively.
 In addition, for a finite set $A$, denote by $\# A$ its cardinality. For each $K \in \mcal T_h$, we use $\bm n_{\partial K}$ to denote the unit outward normal vector of $\partial K$ , which will be abbreviated as $\bm n$ if it does not cause any confusion. For an edge $e \in \mathcal E(\mathcal T_h)$, let $\bm n_e$ be its unit normal vector, which will be abbreviated
 as $\bm n$. Consider two adjacent triangles, $K_1$ and $K_2$ sharing an interior edge $e$. Define the jump of a function $v$ on $e$ as
 \begin{align*}
 	[\![ v]\!] |_e \coloneqq   (v|_{K_1})|_e{\bm n}_e\cdot {\bm n}_{\partial K_1} + (v|_{K_2})|_e{\bm n}_e\cdot {\bm n}_{\partial K_2}%
 		.
 	 \end{align*}
 On a face $e$ lying on the boundary $\partial \Omega$, the jump becomes $	[\![ v]\!] |_e = v|_e$.

 The weak formulation of problem \eqref{model:GEK} is to find $w\in V\coloneqq H^3_0(\Omega)$ such that
 \begin{equation}\label{weakForm:GEK}
 	\iota^2a(w,v) + b(w,v) = (f,v)\quad \forall\;v\in V,
 \end{equation}
 where
 \[
 a(w,v) =   (\nabla^3 w, \nabla^3 v) \quad\text{and}\quad  b(w ,v) =  (\nabla^2   w , \nabla^2 v).%
 \]
 We recall the following regularity result for problem \eqref{weakForm:GEK} from \cite{ChenHuangHuang2026GEKP}:
 \begin{theorem}\label{th:GEKRegularity}
 	Assume $\Omega$ is a convex polygon. Let $w\in H^3_0(\Omega)$ and $w_0\in H^{2}_0(\Omega)$ be the solutions of problems \eqref{weakForm:GEK} and \eqref{weakForm:KPB}, respectively. Then it holds $w\in H^{4}(\Omega)$ and
 	\begin{equation}\label{eq:GEKRegularity}
 		|w-w_0|_{2} + \iota\|w\|_3 + \iota^2\|w\|_{4}\lesssim \iota^{1/2}\|f\|_0.
 	\end{equation}
 	Here and hereafter, the hidden constant is independent of the size parameter $\iota$.
 \end{theorem}

 Set $\iota=0$, problem \eqref{model:GEK} will reduce to a fourth-order elliptic problem:
 \begin{equation}\label{model:KPB}
 	\begin{cases}
 		\Delta^2 w_0  = f  & \textrm{in}~\Omega ,\\
 		w_0 = \partial_{\bm{n}} w_0 = 0&\textrm{on}~\partial\Omega .
 	\end{cases}
 \end{equation}
 Its weak formulation is to find $w_0\in  H^2_0(\Omega)$ such that
 \begin{equation}\label{weakForm:KPB}
 	b(w_0 ,v) = (f,v)\quad \forall\;v\in H^2_0(\Omega).
 \end{equation}
 As stated in Corollary~7.3.2.5 in \cite{Grisvard1985}, for any convex polygon $\Omega$, one has $w_0\in H^3(\Omega)$ and
 \begin{equation}\label{eq:regularityKPB}
 	\|w_0\|_3\leq C\|f\|_0.
 \end{equation}
 with a positive constant $C$.

 \section{Discrete form}\label{SectDiscreteF}
 We begin by introducing the notation that will be used. Given a triangle $K$ with vertices $\bm x_i$, denote by $  e_i$ $(1 \leq i \leq	3)$ the edges of $K$ without $\bm x_i$ as its vertex, and
 by $\lambda_1$, $\lambda_2$, $\lambda_3$ the barycentric coordinates of $K$. Denote by $|K|$ and $|e_i|$ the
 measures of $|K|$ and $|e_i|$, respectively.
 Let $b_k$ be the bubble function defined by $b_k = \lambda_1\lambda_2\lambda_3$.
 \subsection{A cubic $H^3$ nonconforming finite element}
   A cubic $H^3$ nonconforming macro-element $(K,V_K,\mathcal N_K^V)$ is constructed in \cite{HuZhang2019}, which is designed for solving a tri-harmonic equation. Besides, the finite element solution converges at the optimal order in $H^3$-norm. The new element is a piecewise cubic polynomial on three triangles, $\bm x_0 \bm x_2\bm x_3$, $\bm x_0 \bm x_3\bm x_1$ and
 $\bm x_0 \bm x_1\bm x_2$. The components of $\mathcal N_K^V$ are shown in Figure \ref{Fig:CubicH3}. In addition, we require some continuity constraints to properly define the finite element, whose precise construction is detailed in Sect. 3 of \cite{HuZhang2019}.  %
 On a single macro-element, the finite element space is defined by
\begin{align}
 V_K =\{ & v_h \in L^2(K)\;|\; v_h |_{K_i} \in \mathbb P_3(K_i),\, K = \cup_{i=1}^3 K_i,\,\text{\rm int}(K_j)\cap\text{\rm int}(K_i) = \emptyset, \text{ the function}\notag\\
&\text{values and the first derivatives of $v_h$ are continuous at the vertexes of all three $K_i$,}\notag\\
&\text{the second derivatives of $v_h$ are continuous across the three internal edges at the}\notag\\
&\text{middle point $m^0_i$ and the sum of three jumps of a third-order scaled directional}\notag\\
&\text{derivative at $m^0_i$ is limited to zero}\}.
 \end{align}
The global space of the cubic $H^3$ nonconforming finite element can be defined abstractly as
 \begin{align}
 	V_h = \{& v_h \in L^2(\Omega)\;|\; v_h |_K\in V_K\quad\forall K\in \mathcal T_h,\notag\\
 	&\partial_{x_1^{\alpha_1}x_2^{\alpha_2}}v_h |_K(x_j) =\partial_{x_1^{\alpha_1}x_2^{\alpha_2}}v_h |_{K'}(x_j)\quad \forall\bm x_j\in \mathcal V^i(\mcal T_h),\,K,K'\in \omega(\bm x_j),\,|\alpha| = 0,1\notag\\
 	&\partial_{\bm {nn}}v_h|_K(\bm m_j)=\partial_{\bm {nn}}v_h|_{K'}(\bm m_j)  \quad \forall\, \bm m_j\in   e_j ,\,  e_j  \in \mathcal E^{i}(\mcal T_h)\notag\\
 	& \partial_{x_1^{\alpha_1}x_2^{\alpha_2}}v_h |_K(\bm x_j)=0\quad \forall\bm x_j \in \mathcal V^{\partial}(\mcal T_h),\,|\alpha| = 0,1\notag\\
 	&\partial_{\bm {nn}}v_h|_K(\bm m_j)=0\quad \forall\, \bm m_j\in   e_j ,\,  e_j  \in \mathcal E^{\partial}(\mcal T_h)\}.
 \end{align}
 Obviously, it holds
 \begin{align*}%
 	\int_e [\![\nabla^2v]\!] = 0 \quad \forall v\in V_h,\,e\in\mathcal E(\mathcal T_h).
 \end{align*}

 \begin{figure}[t]
 	\centering
 	\includegraphics[scale=0.3]{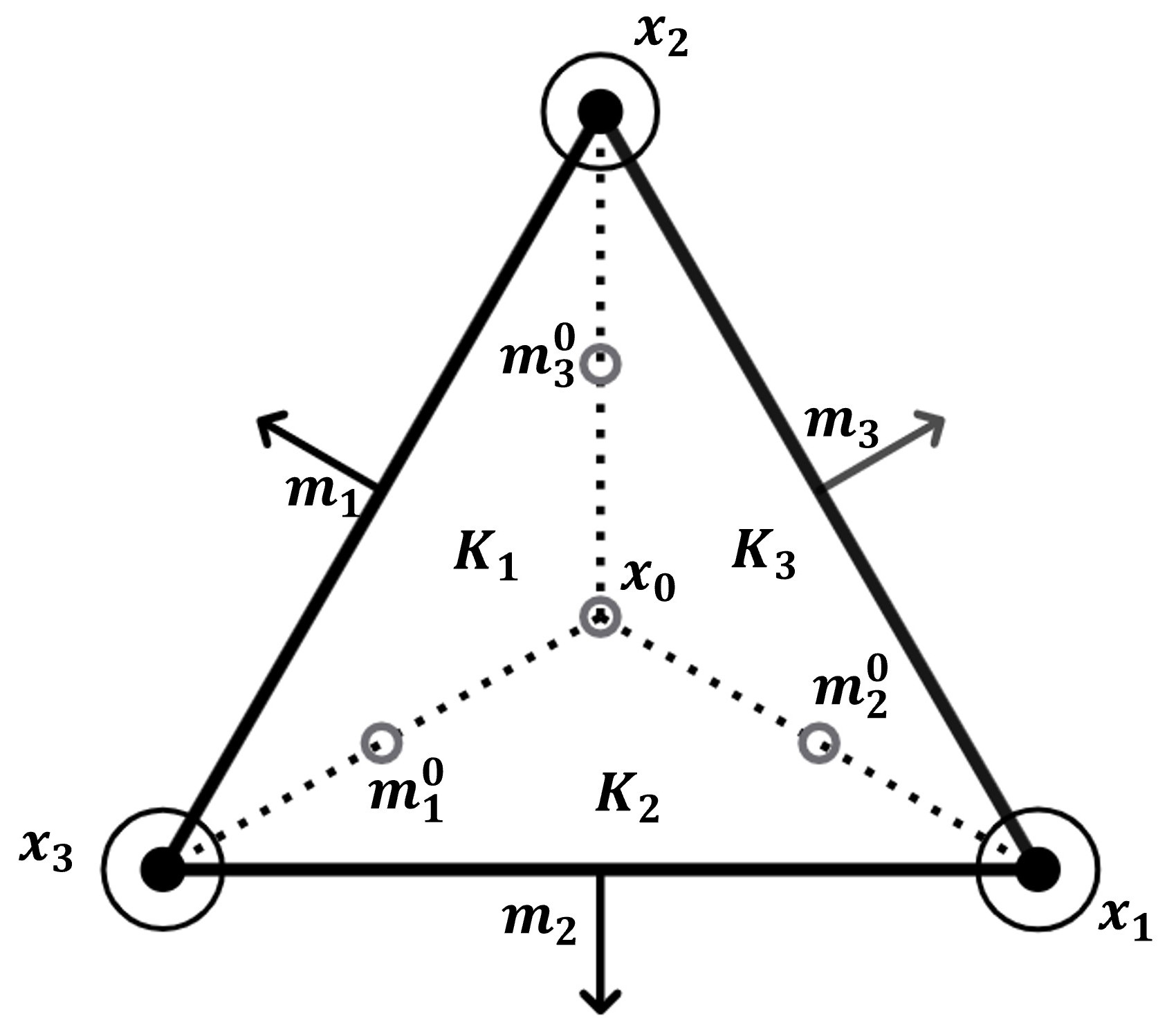}
 	\caption{$\mathcal N_K^V$:
 		``$\bullet$" stands for function value at the point, ``$\bigcirc$" indicates first-order derivative evaluation at the point, ``$\to$" the second normal derivative at the outer mid-edge points.}
 	\label{Fig:CubicH3}
 \end{figure}
 Let $I^V_h$ be the quasi interpolation operator from $H^3_0(\Omega )$ to $V_h$. Then, for any $v\in H^3_0(\Omega)$, the DoFs of $I_h^V v$ are given as:
 \begin{align}
 I^V_hv(\zeta)  & =   v(\zeta )  ,
 	\label{weakinterpolationCubicH31}  \\
 	\nabla I^V_h v(\zeta)  &= \frac {1}{\#\mcal T_{\zeta}} \sum_{K'\in \mcal T_{\zeta}}   (\nabla \Pi_{K'}^3v)(\zeta),
 	\label{weakinterpolationCubicH32}\\
 	\partial_{\bm {nn}} I^V_h  v(\zeta_e)& = \frac {1}{\#\mcal T_{e}} \sum_{K'\in \mcal T_{\zeta}}   \partial_{\bm {nn}} \Pi_{K'}^3v(\zeta_e ) ,
 	\label{weakinterpolationCubicH33}
 \end{align}
 where $\zeta_e$ is the midpoint of $e$, $\zeta\in \mathcal V^i(\mathcal T_h)$, $e_j  \in \mathcal E^{i}(\mcal T_h)$, $K\in \mathcal T_h$, $ \Pi_{K}^3$ is the standard $L^2$ projection onto $\mathbb P^3(K)$ and DoFs \eqref{weakinterpolationCubicH31}-\eqref{weakinterpolationCubicH33} vanish on boundary. Following the ideas for proving Theorem 3.5.3 in \cite{Shi-Wang-2013}, we have
 \begin{lemma}
  \begin{align}\label{eq:interpolationCubicH3}
 	 	|v-I_h^Vv |_m &\lesssim h^{s-m} |v|_s\quad \forall v\in H_0^3(\Omega)\cap H^s(\Omega),\, 0\leq m\leq s,\,s = 3,4.
 	 \end{align}
 \end{lemma}
\begin{proof}
 For any $ \bm v\in H_0^3(\Omega)\cap H^s(\Omega)$ with $s = 3,4$, it follows from the scaling argument
\begin{align} \label{eq:interpoNormequi}
	&\quad\; h_K^{2m}| \Pi_{K}^3 v - I^V_h v |_{m,K}^2\lesssim  |  \Pi_{K}^3 v - I_h v |_{0,K}^2
	\notag\\
	&\lesssim h_K^2\sum_{\delta\in \mathcal{V}(K)}\big| \Pi_{K}^3 v  - (I_h  v)|_K\big|^2(\delta) +  h_K^4\sum_{\delta\in \mathcal{V}(K) }\big|\nabla( \Pi_{K}^3 \bm v  - (I_h  v)|_K)\big|^2(\delta)
	\notag\\
	&\quad+ h_K^6 \sum_{\delta\in \mathcal{V}(K)} | 	\partial_{\bm {nn}} \Pi_{K}^3  v  - 	\partial_{\bm {nn}}(I_h  v)|_K |^2  (\delta)
	\notag\\
	&\lesssim h_K^2\sum_{\delta\in \mathcal{V}(K)}\big| \Pi_{K}^3 v  -  v \big|^2(\delta) +  \sum_{\delta\in \mathcal V(K)} \sum_{e\in \mathcal E(\mathcal T_h) ,\delta\in\partial  e}  h_K^3\|\llbracket \nabla  \Pi_{K}^3 v \rrbracket\|_{0,e}^2  \notag\\
	&\quad+ h_K^5  \sum_{\delta\in \mathcal V(K)} \sum_{e\in \mathcal E(\mathcal T_h) ,\delta\in\partial  e}  \|\llbracket   \partial_{\bm n\bm n}  \Pi_{K}^3 v \rrbracket \|_{0,e}^2.
\end{align}
Employing the trace inequality, the Sobolev inequality, \eqref{eq:interpoNormequi} and the error estimates of $\Pi_K^3$,
\begin{align*} %
	&\quad\; h_K^{2m}|  v - I^V_h v |_{m,K}^2 \leq  h_K^{2m}| v-\Pi_{K}^3 v |_{m,K}^2 + h_K^{2m}| \Pi_{K}^3 v - I^V_h v |_{m,K}^2\lesssim \sum_{\delta\in \mathcal V(K)}  h_K^{2s}|v|_{s,\omega_{\delta}}.
\end{align*}
The proof is completed.
\end{proof}

  \subsection{A new Zienkiewicz-type nonconforming element}

 A NZT finite element for the biharmonic equation is proposed in \cite{WangShiXu2007NZT}. For $1\leq i<j\leq 3$, we define
 \begin{align}
 	b_{ij}=\lambda_i^2 \lambda_j-\lambda_i \lambda_j^2
 	+ 3(\frac 23 (\lambda_i - \lambda_j) + \sum_{1\leq k\leq 3\atop k\not=i,k\not=j} \frac{(\nabla \lambda_i - \nabla\lambda_j)^\intercal\nabla\lambda_k}{\|\nabla\lambda_k\|^2}(n\lambda_k -1))b_k.
 \end{align}
 The NZT finite element is defined by $(K, \mathcal P_K, \mcal N_K^{NZT})$ with
 	 \[ \mathcal P_K = P_2(K) + \text{span}\{ b_{ij} | 1 \leq i < j \leq 3\} .\]
  The components of $\mcal N_K$ are shown in Figure \ref{Fig:NZT}.
 \begin{figure}[t]
 	\centering
 	\includegraphics[scale=0.1]{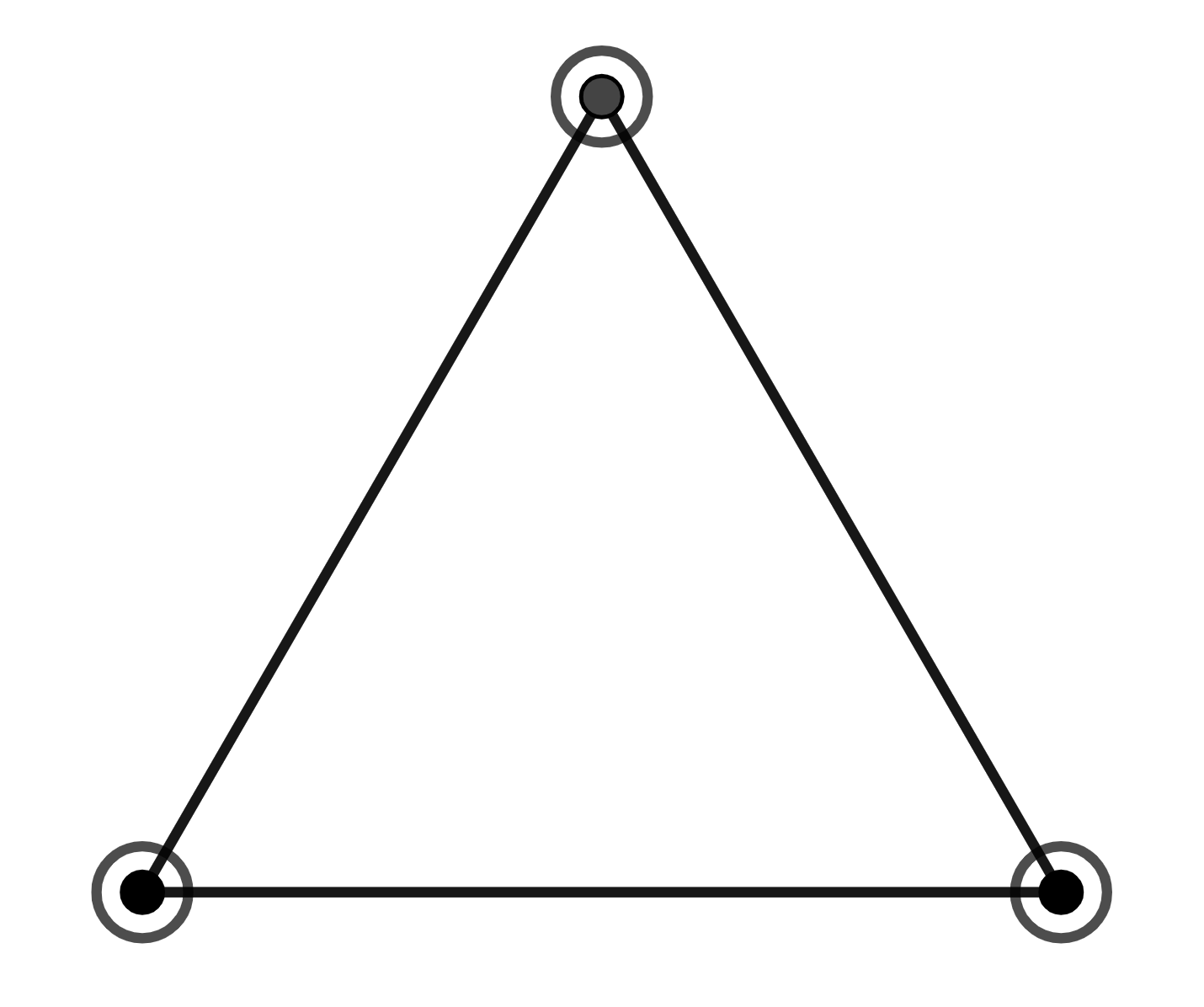}
 	\caption{
 	$\mathcal N_K^{NZT}$:
 	``$\bullet$" stands for function value at the point, ``$\bigcirc$" indicates first-order derivative evaluation at the point.
 	}
 	\label{Fig:NZT}
 \end{figure}
 The global space of NZT element reads as
 \begin{align*}
 	V_h^{NZT} = \{& v\in L^2(\Omega)\,|\,v|_K\in \mathcal P_K , K\in \mcal T_h, \text{ $v$ and $\nabla v$ are continuous at all vertices of}\\
 	& \text{\quad\quad\quad elements in $\mcal T_h$, $v$ and $\nabla v$ vanish at all vertices belonging to $\partial \Omega$}\}.
 \end{align*}
 Obviously, it holds
 \begin{align*}%
 	\int_e [\![\nabla v]\!] = 0   \quad \forall v\in V_h^{NZT},\,e\in\mathcal E(\mathcal T_h).
 \end{align*}
 Let $I_h^{NZT}$ be the quasi interpolation operator from $H^2_0(\Omega )$ to $V_h$. Then, for any $v\in H^2_0(\Omega)$, the DoFs of $I_h^{NZT} v$ are given as:
 \begin{align}
 	I_h^{NZT}v(\zeta)  & = \frac {1}{\#\mcal T_\zeta}\sum_{K'\in \mcal T_{\zeta}}   v(\zeta )  ,
 	\label{weakInterpolation1}  \\
 	\nabla I_h^{NZT} v(\zeta)  &= \frac {1}{\#\mcal T_{\zeta}} \sum_{K'\in \mcal T_{\zeta}}   (\nabla \Pi_{K'}^3v)(\zeta),
 	\label{weakInterpolation2},
 \end{align}
 where $\zeta\in \mathcal V^i(\mathcal T_h)$, $K\in \mathcal T_h$, $ \Pi_{K}^3$ is the standard $L^2$ projection onto $\mathbb P^3(K)$ and DoFs \eqref{weakInterpolation1}-\eqref{weakInterpolation2} vanish on boundary. Following the ideas for proving Theorem 3.5.3 in \cite{Shi-Wang-2013}, we have
 \begin{equation}\label{eq:interpolationNZT}
 	|v-I_h^{NZT} v|_{m,h}\lesssim h ^{s-m}|v|_{s  }\quad\forall \; v\in H^{s}(\Omega)\cap H^2_0(\Omega),\, 0\leq m \leq s ,\, 2\leq s\leq 3.
 \end{equation}

 \subsection{$H^3$ nonconforming finite element method}\label{SectFEM}

 Inspired by the approach in \cite{WangXuHu2006}, we design a robust nonconforming finite element method for a sixth-order elliptic
 singular perturbation problem \eqref{weakForm:GEK} by use of the macro-element $(K,V_K,\mathcal N_K^V)$. It also use the NZT element to discrete the lower part of the bilinear form.

 The discrete form of problem \eqref{weakForm:GEK} reads as: Find $w_h\in V_h$ such that
 \begin{equation}\label{discreteForm:GEK}
 	\iota^2a_h(w_h,v_h) + b_h(I_h^{NZT}w_h,I_h^{NZT}v_h) = (f,I_h^{NZT}v)\quad \forall\;v\in V_h,
 \end{equation}
 where
 \[
 a_h(w,v) =  \sum_{K\in \mcal T_h} (\nabla^3 w, \nabla^3 v)_k \quad\text{and}\quad  b(w ,v) = \sum_{K\in \mcal T_h} (\nabla^2   w , \nabla^2 v)_k.%
 \]
 Define
  \[
 \|v\|_{\iota,h}^2 =	\iota^2a_h(v,v) + b_h(I_h^{NZT}v,I_h^{NZT}v)= \iota^2\sum_{K\in \mcal T_h}|v|_{3,K}^2 + \sum_{K\in \mcal T_h}|I_h^{NZT}v|_{2,K}^2 .
 \]
Easily, we can conclude from the Lax-Milgram lemma \cite{Ciarlet1978FEM} that the discrete problem \eqref{discreteForm:GEK} has a unique solution when $\iota>0$.

 \section{Error analysis}\label{Sect:ErrorAnalysis}
 \begin{lemma}	\label{lemma:interpol}
 	Assume $\Omega$ is a convex polygon. Let $w\in V$ and $w_0\in H^2_0(\Omega)\cap H^3(\Omega)$ be the solutions of problems \eqref{weakForm:GEK} and \eqref{weakForm:KPB}, respectively. Then, it holds
 	\begin{align}
 		\|w - I_h^Vw\|_{\iota,h}&\lesssim \iota^{1/2}\|f\|_0
 		 \label{eq:errorwIhw}\\
 		|w - I_h^{NZT} w|_{2,h}&\lesssim h|w_0|_3 + \iota^{1/2}\|f\|_0.\label{eq:errorwIhNZTw}
 	\end{align}
 \end{lemma}
 \begin{proof}
 	Since $\mathcal N_K^{NZT}\subsetneq \mathcal N_K^V$, it follows from the definitions of $I_h^V$ and $I_h^{NZT}$ that
 	\[
 	I_h^{NZT}I_h^Vw = I_h^{NZT} w.
 	\]
 Then, we obtain from \eqref{eq:interpolationCubicH3} and \eqref{eq:GEKRegularity} that
 		\begin{align*}
 		\|w - I_h^Vw\|_{\iota,h}
 		=& \iota |w - I_h^Vw|_{3,h}  +  	|I_h^{NZT}(w - I_h^Vw)|_{2,h}
 		\lesssim  \iota|w|_3  + |w|_2\lesssim \iota^{1/2}\|f\|_0.
 	\end{align*}
 	Next, by \eqref{eq:interpolationNZT} and \eqref{eq:GEKRegularity}, it holds
 	\begin{align*}
 			|w - I_h^{NZT} w|_{2,h} =&	|w_0 -I_h^{NZT} w_0|_{2,h} +	|w -w_0 -  I_h^{NZT} (w-w_0)|_{2,h}
 			\notag\\
 			\lesssim&h|w_0|_3 +|w-w_0|_2
 			\notag\\
 			\lesssim&h|w_0|_3 + \iota^{1/2}\|f\|_0.
 	\end{align*}
 \end{proof}
 \begin{lemma}\label{lm:bhf}
 		Under the assumptions of Lemma \ref{lemma:interpol}, it follows
 	\begin{align}\label{bhfrobust}
		\iota^2a_h(  w,v_h) +	b_h(w  , I_h^{NZT}  v_h)  -(f,I_h^{NZT} v_h)   \lesssim& (\iota^{1/2} \|f\|_0  + h|w_0|_{3})\|v_h\|_{\iota,h}\quad  v_h\in V_h  .%
 	\end{align}
 \end{lemma}
 \begin{proof}
 	Assume $w_{h0}$ is the solution of the following problem:  Find $ w_{h0}\in V_h^{NZT}$ such that
 	\begin{align} \label{eq:NZTKPB}
 		b_h(  w_{h0} ,    v_h) = (f,  v_h)\quad \forall v_h \in V_h^{NZT}.
 	\end{align}
 	Moreover, we have from Theorem 2 in \cite{WangShiXu2007NZT} that
 	\begin{align*}%
 		|w_0 -  w_{h0}|_2 \lesssim h|w_0|_3 .
 	\end{align*}
Then,
 	\begin{align}\label{eq:errorbhf}
 		b_h(w  , I_h^{NZT}  v_h)  -(f,I_h^{NZT} v_h)
 		=& b_h(  w  , I_h^{NZT} v_h) -b_h(w_{h0} , I_h^{NZT} v_h)\notag\\
 		=& b_h(  w - w_{0} ,I_h^{NZT} v_h)-b_h(  w_0 - w_{h0} ,I_h^{NZT} v_h)\notag\\
 		\lesssim& (| w - w_{0}|_2 +h|w_0|_3  )\|v_h\|_{\iota,h} .
 	\end{align}
 		Finally, combining \eqref{eq:errorbhf}, the Cauchy-Schwarz inequality and \eqref{eq:GEKRegularity}, we can obtain the estimate \eqref{bhfrobust}.
 \end{proof}
 	\begin{lemma}\label{lm:ahbh}
 		Under the assumptions of Lemma \ref{lemma:interpol}, it holds that for all $v_h \in V_h$
 		\begin{align*}
			\iota^2a_h( I_h^Vw-w,v_h) + b_h(I_h^{NZT}(I_h^V w-w)  , I_h^{NZT} v_h)\lesssim( \iota^{1/2} \|f\|_0+ h|w_0|_{3})\|v_h\|_{\iota,h}%
 		\end{align*}
 	\end{lemma}
 	\begin{proof}
 		Easily,
 	\begin{align}\label{eq:ahbh0}
 		&	\iota^2a_h( I_h^Vw-w,v_h) + b_h(I_h^{NZT}I_h^V w-w  , I_h^{NZT} v_h) \notag\\
 		\lesssim&  (\|w - I_h^Vw\|_{\iota,h}  + |w-I_h^{NZT}w|_{2,h}) \|v_h\|_{\iota,3,h}.
 	\end{align}
 	which along with Lemma \ref{lemma:interpol} and \eqref{eq:GEKRegularity} implies the desired result.%
 	\end{proof}
 	\begin{theorem}\label{convergence}
 			Assume $\Omega$ is a convex polygon. Let $w\in V$, $w_h\in V_h$ and $w_0\in H^3(\Omega)\cap H^2_0(\Omega)$ be the solutions of problems \eqref{weakForm:GEK}, \eqref{discreteForm:GEK} and \eqref{weakForm:KPB}, respectively. Then it holds
 		\begin{align}\label{eq:error}
 			\|w - w_h\|_{\iota,h}+\| I_h^Vw - w_h\|_{\iota,h}\lesssim \iota^{1/2} \|f\|_0  + h|w_0|_{3} .
 		\end{align}
 	\end{theorem}
 	\begin{proof}
 		Let $v_h = I_h^Vw-w_h \in V_h$. We have
 		\begin{align}\label{eq:Ihwwh}
 			\| I_h^Vw - w_h\|_{\iota,h}^2= &   \iota^2a_h( I_h^Vw -w_h,v_h) + b_h(I_h^{NZT}I_h^V w -I_h^{NZT}w_h ,I_h^{NZT}v_h) \notag\\
 			= &   -(f,I_h^{NZT} v_h) +\iota^2a_h( I_h^Vw,v_h) + b_h(I_h^{NZT}I_h^V w  , I_h^{NZT} v_h)  \notag\\
 			= &  -(f,I_h^{NZT} v_h) +\iota^2 a_h(w,v_h) + b_h(I_h^{NZT}w  , I_h^{NZT}  v_h)  \notag\\
 			&+\iota^2a_h( I_h^Vw-w,v_h) + b_h(I_h^{NZT}(I_h^V w-w)  , I_h^{NZT} v_h) .
 		\end{align}
 		Employing  Lemma \ref{lm:bhf} and Lemma \ref{lm:ahbh}, we have from \eqref{eq:Ihwwh} that
 		\[
 		\| I_h^Vw - w_h\|_{\iota,h}\lesssim \iota^{1/2} \|f\|_0 + h|w_0|_{3}  .
 		\]
 Lastly, by combining the triangle inequality, \eqref{eq:errorwIhw} and the previous bound, we establish the desired result.
 	\end{proof}
 	According to Theorem \ref{convergence}, \eqref{eq:errorwIhNZTw}, \eqref{eq:regularityKPB} and \eqref{eq:GEKRegularity}, we can easily derive the following corollary.
 	\begin{corollary}
 			Under the assumptions of Theorem \ref{convergence}, we have
 		\begin{align*}%
 			\|w_0 - w_h\|_{\iota,h}\leq \iota^{1/2}\|f\|_0   + h|w_0|_3.
 		\end{align*}
 	\end{corollary}
 	\begin{corollary}\label{medius}
  	Let $w\in V$ and $w_h\in V_h$ be the solutions of problems \eqref{weakForm:GEK} and \eqref{discreteForm:GEK}, respectively. Then we have
 		\begin{equation*}
 			\|w-w_h\|_{\iota,h}\to 0,\quad h\to 0.
 		\end{equation*}
 		\end{corollary}
 		\begin{figure}[t]
 		\centering
 		\includegraphics[scale=0.25]{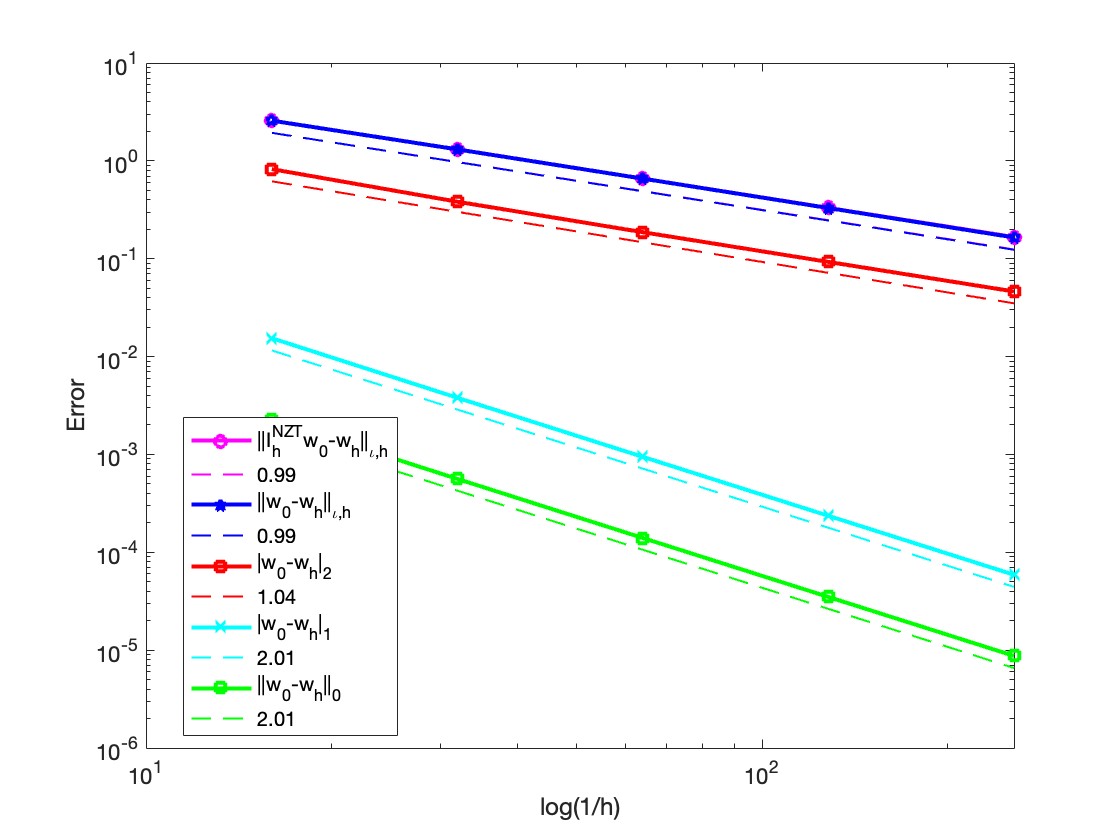}
 		\caption{
 			The performance of the discrete problem \eqref{discreteForm:GEK} with $\iota=10^{-8}$.}
 		\label{Fig:GEK2}
 	\end{figure}

 	\section{Numerical experiments}
 	To validate the theoretical findings, we present a numerical experiment in this section, with the problem data adopted from \cite{ChenHuangHuang2026GEKP}. 
 		This example to testify two parameter robustness with respect to $\lambda$ and $\iota$. Set $\Omega=[0,1]\times[0,1]$. The exact solution of the reduced problem \eqref{model:KPB} is set to be a function in the form
 		\begin{align*}
 			w_0=& \sin^2(\pi x)\sin^2(\pi y).
 		\end{align*}
 		Set $f$ computed from \eqref{model:KPB} to be the right side term of \eqref{model:GEK}. It's easy to check that $ f$ is independent of $\iota$. We can see from Figure~3 in \cite{ChenHuangHuang2026GEKP} that the exact solution of \eqref{model:GEK} $w$ possesses strong boundary layer. We use this example to demonstrate the robustness of the proposed method \eqref{discreteForm:GEK} with respect to size parameter $\iota$.
 	\begin{table}[!htb]%
 		\small
 		\centering
 		\caption{Errors of the discrete problem \eqref{discreteForm:GEK}.}
 		\begin{tabular}{ccccccccccc}
 			\toprule
 			\multirow{2}{*}{$\iota$} &\multirow{2}{*}{Error} & \multicolumn{5}{c}{$h$} \\
 				\cline{3-7}
 		 &&   $ 1/16$  & $ 1/32$  & $ 1/64$   & $ 1/128$ & $ 1/256$\\
 			\midrule[1pt]
 			\multirow{5}{*}{$10^{-6}$}
 		&	$ \|I_h^{NZT}w_0- w_h\|_{\iota,h}$ &2.578e+00 & 1.313e+00    &  6.579e-01  & 3.293e-01 &1.647e-01
 			\\
 			&	$ \|w_0- w_h\|_{\iota,h}$  &2.588e+00  & 1.332e+01    &  6.995e-01  & 4.059e-01 & 2.887e-01
 			\\
 		&	$|w_0- w_h|_{1}$ & 1.541e-02  & 3.791e-03 &9.438e-04  &2.357e-04    &5.899e-05 \\
 		&	$|w_0- w_h|_{2,h}$ & 8.231e-01 &3.824e-01 &1.866e-01&9.269e-02    &4.626e-02  \\
 		& 	$\|w_0- w_h\|_0$  &2.285e-03& 5.622e-04&1.399e-04 &3.494e-05  &  8.750e-06  \\
 			\midrule[1pt]
 			\multirow{5}{*}{$10^{-8}$}
 			&	$ \|I_h^{NZT}w_0- w_h\|_{\iota,h}$  &2.577e+00 & 1.311e+00    &  6.584e-01  & 3.302e-01 &1.664e-01
 		\\
 		&	$ \|w_0- w_h\|_{\iota,h}$  & 2.577e+00  & 1.310e+00    &  6.580e-01  &  3.294e-01 & 1.649e-01
 		\\
 		&	$|w_0- w_h|_{1}$ & 1.541e-02  & 3.791e-03 &9.438e-04  &2.357e-04    &5.899e-05 \\
 		&	$|w_0- w_h|_{2,h}$ & 8.231e-01 &3.824e-01 &1.866e-01&9.269e-02    &4.626e-02  \\
 		& 	$\|w_0- w_h\|_0$  &2.285e-03& 5.622e-04&1.399e-04 &3.494e-05  &  8.750e-06  \\
 		\midrule[1pt]
 		\multirow{5}{*}{$10^{-10}$}
 		&	$ \|I_h^{NZT}w_0- w_h\|_{\iota,h}$ &2.577e+00 & 1.310e+00    &  6.579e-01  & 3.293e-01 &1.647e-01
 	\\
 	&	$ \|w_0- w_h\|_{\iota,h}$  &2.577e+00  & 1.310e+01    &  6.580e-01  &  3.293e-01 & 1.647e-01
 	\\
 		&	$|w_0- w_h|_{1}$ & 1.541e-02  & 3.791e-03 &9.438e-04  &2.357e-04    &5.895e-05 \\
 		&	$|w_0- w_h|_{2,h}$ & 8.231e-01 &3.824e-01 &1.866e-01&9.269e-02    &4.626e-02  \\
 		& 	$\|w_0- w_h\|_0$  &2.285e-03& 5.622e-04&1.399e-04 &3.494e-05  &  8.740e-06  \\
 			\midrule[1pt]
 		\multirow{2}{*}{$0$}
 		&$ \|I_h^{NZT}w_0- w_{h0}\|_{2,h}$  &2.577e+00  & 1.310e-02 & 6.580e-01  & 3.293e-01 &1.647e-01    \\
 		& 	$\|w_0- w_{h0}\|_0$  &2.050e-03&5.068e-04&1.262e-04 &3.149e-05  & 7.866e-06  \\
 		\bottomrule
 		\end{tabular}
 		\label{Tab:GEK22}
 	\end{table}

 	Due to the fact that the mesh we used here is quasi-uniform, it follows from the trace inequality, \eqref{eq:GEKRegularity} and \eqref{eq:regularityKPB} that
 	\begin{align}\label{eq:errorw0w}
 		\|w-w_0\|_{\iota,h} + \|w_0-I_h^{NZT}w_0\|_{\iota,h}\leq \iota^{1/2}\|f\|_0.
 	\end{align}
 	Then, it suffices to observe the behavior of error $
 I_h^{NZT} w_0-w_h $ to verify \eqref{eq:error}.

 The numerical errors with different $h$ and $\iota$ are listed in Table \ref{Tab:GEK22}, from which we can see that the proposed method converges for small $\iota$. It is noted that the numerical solution of \eqref{discreteForm:GEK} converges to the numerical solution of \eqref{eq:NZTKPB} in $H^2$ semi-norm and $L^2$ norm as $\iota\to0$. As observed from Figure \ref{Fig:GEK2}, error $ \|I_h^{NZT}w_0- w_h\|_{\iota,h}$ has linear convergence with respect to the mesh size $h$, while  exhibits optimal convergence. These results combined with \eqref{eq:errorw0w} are consistent with the estimate \eqref{eq:error}. Moreover, we can see from Figure \ref{Fig:GEK2} that $|w_0- w_h|_{2}\eqsim O(h^{1.04})$, $|w_0- w_h|_{1}\eqsim O(h^{2.01})$ and $|w_0- w_h|_{0}\eqsim O(h^{2.01})$.

 {\bf Acknowledgement.} The second author is grateful to Prof. Jun Hu from Peking University who brought his attention to the reference \cite{HuZhang2019} during an international conference in Shanghai University in 2026.

\section*{Reference}
\bibliographystyle{abbrv}
\bibliography{GEK}
\end{document}